\documentclass[10pt,reqno]{amsart}
\usepackage{amssymb}
\usepackage{amsmath}
\usepackage{amsthm}
\usepackage{graphicx}

\usepackage{subfigure}

\usepackage{caption}

\usepackage{bm}

\usepackage{enumitem}

\usepackage[colorlinks=true,urlcolor=blue,
citecolor=red,linkcolor=blue,linktocpage,pdfpagelabels,
bookmarksnumbered,bookmarksopen]{hyperref}

\newcommand{\mihref}[3][blue]{\href{#2}{\color{#1}{#3}}}

\usepackage{ifthen} 

\provideboolean{shownotes} 
\setboolean{shownotes}{true} 
\newcommand{\margnote}[1]{
\ifthenelse{\boolean{shownotes}}%
{\marginpar{\raggedright\tiny\texttt{#1}}}%
{}%
}

\newcommand{\hole}[1]{
\ifthenelse{\boolean{shownotes}}%
{\begin{center} \fbox{ \rule {.25cm}{0cm}
\rule[-.1cm]{0cm}{.4cm} \parbox{.85\textwidth}{\begin{center}
\texttt{#1}\end{center}} \rule {.25cm}{0cm}}\end{center}}
{}
}

\usepackage[titletoc,title]{appendix}

\numberwithin{equation}{section} 

\newcounter{cont}[section]

\newtheorem{theorem}[cont]{Theorem}

\newtheorem{lem}[cont]{Lemma}

\newtheorem{corollary}[cont]{Corollary}
\newtheorem{lemma}[cont]{Lemma}

\theoremstyle{definition}
\newtheorem{definition}[cont]{Definition}

 \theoremstyle{remark}
 \newtheorem{rem}[cont]{Remark}

\renewcommand{\Re}{\mathrm{Re}\,} 
\renewcommand{\Im}{\mathrm{Im}\,}

\newcommand{\e}{\varepsilon}

\newcommand{\N}{\mathbb{N}}
\newcommand{\R}{\mathbb{R}}
\newcommand{\C}{\mathbb{C}}

\newcommand{\cL}{{\mathcal{L}}}
\newcommand{\cI}{{\mathcal{I}}}
\newcommand{\cO}{{\mathcal{O}}}

\newcommand{\ess}{\sigma_\mathrm{\tiny{ess}}}
\newcommand{\ptsp}{\sigma_\mathrm{\tiny{pt}}}

\begin{document}

\title[Spectral stability of small-amplitude dispersive shocks in QHD]{Spectral stability of small-amplitude dispersive shocks in quantum hydrodynamics with viscosity}

\author[R. Folino]{Raffaele Folino}

\address[R. Folino]{Departamento de Matem\'aticas y Mec\'anica\\Instituto de 
Investigaciones en Matem\'aticas Aplicadas y en Sistemas\\Universidad Nacional Aut\'onoma de 
M\'exico\\Circuito Escolar s/n, Ciudad Universitaria C.P. 04510 Cd. Mx. (Mexico)}

\email{folino@mym.iimas.unam.mx}

\author[R.G. Plaza]{Ram\'on G. Plaza}

\address[R. G. Plaza]{Departamento de Matem\'aticas y Mec\'anica\\Instituto de 
Investigaciones en Matem\'aticas Aplicadas y en Sistemas\\Universidad Nacional Aut\'onoma de 
M\'exico\\Circuito Escolar s/n, Ciudad Universitaria C.P. 04510 Cd. Mx. (Mexico)}

\email{plaza@mym.iimas.unam.mx}

\author[D. Zhelyazov]{Delyan Zhelyazov}

\address[D. Zhelyazov]{Departamento de Matem\'aticas y Mec\'anica\\Instituto de 
Investigaciones en Matem\'aticas Aplicadas y en Sistemas\\Universidad Nacional Aut\'onoma de 
M\'exico\\Circuito Escolar s/n, Ciudad Universitaria C.P. 04510 Cd. Mx. (Mexico)}

\email{delyan.zhelyazov@iimas.unam.mx}




\begin{abstract} 
A compressible viscous-dispersive Euler system in one space dimension in the context of quantum hydrodynamics is considered. The dispersive term is due to quantum effects described through the Bohm potential and the viscosity term is of linear type. It is shown that small-amplitude viscous-dispersive shock profiles for the system under consideration are spectrally stable, proving in this fashion a previous numerical observation by Lattanzio \emph{et al.} \cite{LMZ20a,LMZ20b}. The proof is based on spectral energy estimates which profit from the monotonicty of the profiles in the small-amplitude regime.
\end{abstract}


\maketitle

\section{Introduction}\label{sec:intro}
In this paper we investigate stability properties of dispersive shock profile solutions to the following quantum
hydrodynamics (QHD) system with linear viscosity:
\begin{equation}\label{QHD-L}
	\begin{cases}
		\rho_t+m_x=0,\\
		m_t+\left(\displaystyle\frac{m^2}{\rho}+p(\rho)\right)_x=\mu m_{xx}+k^2\rho\left(\displaystyle\frac{(\sqrt{\rho})_{xx}}{\sqrt{\rho}}\right)_x.
	\end{cases}
\end{equation}
Here, the positive constants $\mu,k>0$ are the viscosity and dispersive coefficients, respectively,
while $\rho\geq0$ is the density, $m=\rho u$ is the momentum, where $u$ stands for the velocity and $p(\rho) = \rho^\gamma$, with $\gamma \geq 1$, is the pressure. The function $(\sqrt{\rho})_{xx}/\sqrt{\rho}$ is known as the (normalized) quantum Bohm potential \cite{Boh52a,Boh52b}, providing the model with a nonlinear third order dispersive term. It can be interpreted as a quantum correction to the classical pressure (stress tensor). The viscosity term, in contrast, is of linear type. The resulting system is used, for instance, in superfluidity \cite{Khlt89}, or in classical hydrodynamical models for semiconductor devices \cite{GarC94}.

Models in QHD represent an equivalent alternative formulation of the Schr\"odinger equation, written in terms of hydrodynamical variables, and structurally similar to the Navier--Stokes equations of fluid dynamics. The first derivation of the QHD equations is due to Madelung \cite{Mdlng27}, during the early times of quantum mechanics and it was a precursor of the de Broglie--Bohm causal interpretation of quantum theory \cite{Boh52a,Boh52b,BHK87}. Since then, quantum fluid models have been applied to describe many physical phenomena, such as the modeling of quantum semiconductors \cite{FrZh93,GarC94}, the dynamics of Bose--Einstein condensates \cite{DGPS99,GrantJ73} and the mathematical description of superfluidity \cite{Khlt89,Land41}, among others. The analysis of QHD models have recently attracted the attention of researchers working in analysis and PDE because of their complexity, relevance in physics and the underlying mathematical challenges; for an abridged list of references, see \cite{AnMa09,DMR05,GJV09,JuMi07} and the works cited therein. An important field of study pertains to the emergence of shock waves dominated by dispersion rather than disspation, known as \emph{dispersive shocks}. In the context of QHD models, purely dispersive shocks were first studied in \cite{GuPi74,Sgdv64} (see also \cite{Gas01,HACCES06,HoAb07} for later developments). Motivated by classical fluid theory, a new approach focuses on the interaction between dispersion and viscosity (see, e.g., \cite{DiMu17,GMOB22,Zhel-preprint}), where viscous-dispersive shocks play a preponderant role.

In recent works, the third author and collaborators have studied the existence and stability properties of viscous-dispersive shock profile solutions to system \eqref{QHD-L} with linear viscosity (cf. \cite{LMZ20a,LMZ20b}) and to a different system with nonlinear viscosity coefficient as well (see \cite{LaZ21b,LaZ21a}). The authors pay attention to the interplay of the dispersive and the viscosity terms present in \eqref{QHD-L}. For instance, it is shown in \cite{LMZ20a} that dispersive profiles solutions to \eqref{QHD-L}, connecting a Lax shock for the underlying Euler system in the absence of viscosity and dispersion, do exist. In the small amplitude (or weak shock) regime, that is, when the end states are sufficiently close to each other, profiles are monotone. The case of large amplitude shocks, in contrast, leads to a global analysis of the associated ODE system and to the existence of oscillatory profiles. In the latter case, the dispersion coefficient plays a more significant role (even though, in their analysis, it remains dominated by the viscosity). For both (small- and large-amplitude) regimes, the authors in \cite{LMZ20a} show that the essential spectrum of the linearized operator around the wave is stable or, in other words, that remains in the stable half plane of complex numbers with non-positive real part, provided that the shocks are subsonic or sonic. Moreover, using Evans function techniques, the authors provide some spectral bounds for the absolute value of the point spectrum. Such bounds allow, in turn, to perform some numerical calculations that yield evidence of the stability of the point spectrum as well (see \cite{LMZ20b}). As a result of their findings, Lattanzio \emph{et al.} conjecture that the viscous-dispersive shocks profile solutions to \eqref{QHD-L}, whose existence is proved in \cite{LMZ20a}, are spectrally stable. 


The purpose of this paper is, therefore, very concrete: to provide an analytical proof that dispersive shock profiles for system \eqref{QHD-L} are spectrally stable in the small-amplitude regime. For that purpose, we generalize previous stability results based on energy estimates in the context of viscous \cite{MN85,BHRZ08} and of viscous-capillar fluids \cite{Hu09,ZLY16}. The first part of this paper is devoted to show that, in the small-amplitude case, profiles for system \eqref{QHD-L} share important features with purely-viscous shock profiles, namely, they are monotone and exponentially decaying. These features allow, in turn, to perform energy estimates on the spectral equations. We formulate the spectral problem for the perturbation in terms of integrated variables, which provide better energy estimates (see \cite{Go86,Go91}). Our proof extends the energy estimate performed by Humpherys \cite{Hu09} for viscous fluids with constant capillarity coefficient, by introducing a novel weighted energy that controls the nonlinear (Bohmian) dispersive term. The result applies to subsonic shocks with sufficiently small amplitude, proving in this way the conjecture of \cite{LMZ20a,LMZ20b} for that case.


The paper is structured as follows. Section \ref{sec:structure} contains a description of the viscous-dispersive shocks, as well as the proof of further properties of small-amplitude profiles which are needed in the stability analysis. In section \ref{secspectprob} we pose the spectral problem and recall previous results. The central section \ref{sec:main} contains the main result of the paper: we derive energy estimates on the point spectral problem, which profit from the integrated formulation and the monotonicity and asymptotic behavior of the small-amplitude profiles. These estimates imply the spectral stability of the shocks. We finish the paper with a discussion on open problems.

\subsection*{On notation} 

We denote the real and imaginary parts of a complex number $\lambda \in \C$ by $\Re\lambda$ and $\Im\lambda$, respectively, as well as complex conjugation by ${\lambda}^*$. Complex transposition of vectors or matrices is indicated by the symbol $A^*$, whereas simple transposition is denoted by the symbol $A^\top$. Standard Sobolev spaces of complex-valued functions on the real line will be denoted as $L^2(\R)$ and $H^m(\R)$, with $m \in \N$, endowed with the standard inner products and norms. Linear operators acting on infinite-dimensional spaces are indicated with calligraphic letters (e.g., $\cL$ and $\cI$).

\section{Structure of small-amplitude dispersive shocks}
\label{sec:structure}

The goal of this section is to prove some important properties of the small dispersive shocks profiles for system \eqref{QHD-L}. We invoke the existence theory from previous works (cf. \cite{LMZ20a,Zhel-preprint}) and prove some additional useful features which will be needed in the stability analysis, such as monotonicity and exponential decay. 

First, let us recall some definitions and basic facts concerning system \eqref{QHD-L}, 
with $\mu,k>0$ and $p(\rho)=\rho^\gamma$, for some $\gamma\geq1$.
The associated Euler system related to \eqref{QHD-L} (that is, when viscosity and dispersion parameters are set to zero) is
\begin{equation}\label{eq:Euler}
	\begin{cases}
		\rho_t+m_x=0,\\
		m_t+\left(\displaystyle\frac{m^2}{\rho}+\rho^\gamma\right)_x=0,
	\end{cases}
\end{equation}
and it can be recast in conservative form, $U_t+F(U)_x=0$, where $U=(\rho,m)^\top$ and 
$$F(U)=F(\rho,m)=\left(m,\displaystyle\frac{m^2}{\rho}+\rho^\gamma\right)^{\top}.$$
A direct computation shows that the Jacobian of $F$ is given by
$$DF(\rho,m)=\begin{pmatrix}
0 & 1\\
-m^2/\rho^2+\gamma\rho^{\gamma-1} & 2m/\rho
\end{pmatrix},$$
and its eigenvalues (characteristic speeds) are
$$\lambda_\pm(\rho,m)=\frac{m}{\rho}\pm\sqrt{\gamma\rho^{\gamma-1}}=:u\pm c_s(\rho),$$
where we introduced the sound speed 
\begin{equation}
\label{eq:sound-speed}
	c_s(\rho):=\sqrt{\gamma\rho^{\gamma-1}}.
\end{equation}
A shock wave with end states $\rho^\pm$, $m^\pm$ and shock speed $s$ satisfies the Rankine--Hugoniot conditions $s(U^+-U^-)=F(U^+)-F(U^-)$, namely,
\begin{align}
	s(\rho^+-\rho^-)&=m^+-m^-, \label{eq:RH1}\\
	s(m^+-m^-)&=\left(\displaystyle\frac{m^2}{\rho}+\rho^\gamma\right)^+-\left(\displaystyle\frac{m^2}{\rho}+\rho^\gamma\right)^-. \label{eq:RH2}
\end{align}
Let us consider a traveling wave profile for system \eqref{QHD-L} of the form
$$\rho(x,t)=P(x-st), \qquad m(x,t)=J(x-st),$$
with speed $s\in\R$ and its limiting end states 
$$\lim_{y\to\pm\infty} P(y)=P^\pm, \qquad \lim_{y\to\pm\infty} J(y)=J^\pm,$$
satisfying the Rankine--Hugoniot conditions \eqref{eq:RH1} and \eqref{eq:RH2}.
In order to write the ODEs for the pair $(P,J)$, we recall that the Bohm potential can be rewritten in the conservative form as
$$\rho\left(\displaystyle\frac{(\sqrt{\rho})_{xx}}{\sqrt{\rho}}\right)_x=\frac12\big(\rho(\log\rho)_{xx}\big)_x.$$
Substituting the profiles $P,J$ into \eqref{QHD-L}, we derive the following system of ODEs
\begin{align}
	-sP'+J'&=0, \label{eq:ODE1}\\
	-sJ'+\left(\displaystyle\frac{J^2}P+P^\gamma\right)'&=\mu J''+\frac{k^2}{2}\big(P(\log P)''\big)', \label{eq:ODE2}
\end{align}
where $P:=P(y)$, $J:=J(y)$ and $' := d/dy$, where $y$ denotes the Galilean variable of translation, $y = x-st$.
By integrating \eqref{eq:ODE1} in $(-\infty,y)$, we obtain
$$J(y)-J^-=sP(y)-sP^-, \qquad \forall\, y\in\R.$$
Likewise, integrating in $(y,\infty)$ yields
$$J^+-J(y)=sP^+-sP(y), \qquad \forall\, y\in\R.$$
As a consequence, it follows that
\begin{equation}
\label{eq:J-P}
	J(y)=sP(y)-A,
\end{equation}
where 
\begin{equation}
\label{expr_A}
	A:=sP^{\pm}-J^{\pm}, 
\end{equation}
as follows from the Rankine--Hugoniot condition \eqref{eq:RH1}.
Substituting \eqref{eq:J-P} into \eqref{eq:ODE2} and proceeding as in \cite{LMZ20a} (see section 2), we finally conclude that the profile $P$ satisfies
\begin{equation}\label{2Dsys}
	P''=\frac{2}{k^2} f(P) - \frac{2 s \mu}{k^2} P' + \frac{P'^2}{P},
\end{equation}
where
\begin{equation*}
	f(P)=P^{\gamma}-(As+B)+\frac{A^2}{P},
\end{equation*}
with $A$ defined in \eqref{expr_A} and
\begin{equation}\label{expr_B}
	B:=-s J^{\pm}+\Big{(}\frac{J^2}{P}+P^{\gamma}\Big{)}^{\pm}.
\end{equation}
The constants $A,B$ in $f(P)$ can be expressed in terms of $P^{\pm}$ to obtain
\begin{equation}\label{f_roots0}
	f(P)= P^\gamma +\frac{P^-P^+}{P}\frac{(P^+)^\gamma-(P^-)^\gamma}{P^+-P^-} -\frac{(P^+)^{\gamma+1}-(P^-)^{\gamma+1}}{P^+-P^-}. 
\end{equation}
Note that $P^\pm$ are zeros of $f(P)$; this fact can be verified by using the relations \eqref{expr_A} and \eqref{expr_B} in the definition of $f$. 
Moreover, $f''(P) > 0$ for $P > 0$ and therefore $P^{\pm}$ are the only positive zeros of $f$.

\subsection{Monotonicity}

As it was previously mentioned, the existence of profiles satisfying \eqref{2Dsys} with end states $P^\pm$ is studied in detail in \cite{LMZ20a,Zhel-preprint}. In particular, it is well known that if the amplitude $|P^+-P^-|$ is sufficiently small, then a heteroclinic solution exists (cfr. Lemma 1 in \cite{LMZ20a}).  The next goal is to prove that the profile $P$ is monotone if the amplitude is sufficiently small (see also \cite{Hu09} for a related discussion in the context of compressible fluids).

\begin{lem}\label{lem:P-prop}
Let $P^->0$.
There exists $\varepsilon_0>0$ such that:
\begin{description}
\item[(i)] If $s>0$ and $0<P^--P^+<\varepsilon_0$, then there exists a heteroclinic trajectory $P$ of \eqref{2Dsys} connecting $[P^-, 0]$ to $[P^+, 0]$, which is monotone decreasing with $P'(y)<0$, for any $y\in\R$.
\item[(ii)] If $s<0$ and $0<P^+-P^-<\varepsilon_0$ then there exists a heteroclinic trajectory $P$ of \eqref{2Dsys} connecting $[P^-, 0]$ to $[P^+, 0]$, which is monotone increasing with $P'(y)>0$, for any $y\in\R$.
\end{description}
In both cases \textbf{\emph{(i)}} and \textbf{\emph{(ii)}}, we have 
\begin{equation}\label{eq:prop-P}
	|P'(y)|\leq C\e^2, \qquad \mbox{ and } \qquad |P''(y)|\leq C\e|P'(y)|, \qquad \; \forall\,y\in\R.
\end{equation}
\end{lem}

\begin{proof}
Let us start with \textbf{(i)}; 
in this case we have $s>0$ and assume that the amplitude of the shock, 
\begin{equation}\label{expr_Pp}
	\e := P^- - P^+ > 0, 
\end{equation}
satisfies $\varepsilon\in(0,\e_0)$. The equation $f'(P) = 0$ has a unique positive root,
\begin{equation}\label{expr_P0}
	P_0 = \Big{(} \frac{P^- P^+}{\gamma} \frac{(P^+)^\gamma-(P^-)^\gamma}{P^+-P^-} \Big{)}^\frac{1}{\gamma + 1},
\end{equation}
with the following  expansion 
\begin{equation}\label{expansion_P0}
	P_0 = P^- - \frac{\varepsilon}{2} + \mathcal{O}(\varepsilon^2). 
\end{equation}
Let us expand $f(P)$ around $P_0$:
\begin{equation}\label{expansion_f}
	f(P) = f(P_0) + f'(P_0)(P-P_0) + \frac{1}{2} f''(P_0)(P-P_0)^2 + \mathcal{O}(|P-P_0|^3), 
\end{equation}
and consider the change of variable,
\begin{equation}\label{changevar}
	P = \varepsilon R + \frac{P^+ + P^-}{2}.
\end{equation}
The value $P = P^+$ corresponds to $R = -\frac{1}{2}$ and $P = P^-$ corresponds to $R = \frac{1}{2}$. 
Therefore, the interval $[P^+, P^-]$ corresponds to $[-\tfrac{1}{2},\tfrac{1}{2}]$ in the variable $R$ and it is independent of $\varepsilon$.
Using \eqref{expr_Pp} and \eqref{changevar} we get
\begin{equation}\label{expr_P}
	P = \varepsilon R + P^- - \frac{\varepsilon}{2}.
\end{equation} 
Now, let us make a change of the independent variable $z = \varepsilon y$ and substitute \eqref{expr_P} into \eqref{2Dsys}. This yields
\begin{equation}\label{2Dsys_R}
	\varepsilon^3 R_{zz} = \frac{2}{k^2} f\Big{(}\varepsilon R + P^- - \frac{\varepsilon}{2}\Big{)} - \frac{2 s \mu}{k^2}\varepsilon^2 R_z + \frac{\varepsilon^4}{\varepsilon R + P^- -\frac{\varepsilon}{2}}(R_z)^2.
\end{equation}
Now, let us consider the expansion \eqref{expansion_f}: 
by using \eqref{expr_Pp} and \eqref{expr_P0}, we deduce
\begin{align}
	f(P_0) &= -\frac{\gamma (\gamma + 1)}{8} (P^-)^{\gamma - 2} \varepsilon^2 + \mathcal{O}(\varepsilon^3), \label{expr_f}\\
	f'(P_0) &= 0, \label{expr_df}\\
	f''(P_0) &= \gamma (\gamma + 1) (P^-)^{\gamma - 2} + \mathcal{O}(\varepsilon). \label{expr_d2f}
\end{align}
Using \eqref{expr_P} and \eqref{expansion_P0},  we get 
\begin{equation}
	P - P_0 = \varepsilon R+ \mathcal{O}(\varepsilon^2).  \label{expr_pmp0}
\end{equation}
Substituting \eqref{expr_f}, \eqref{expr_df}, \eqref{expr_d2f} and \eqref{expr_pmp0} into \eqref{expansion_f}, we end up with
\begin{equation}
	f(P) = -\frac{\gamma (\gamma + 1)}{8} (P^-)^{\gamma - 2} \varepsilon^2 + \frac{\gamma (\gamma + 1)}{2} (P^-)^{\gamma - 2} \varepsilon^2 R^2 + \mathcal{O}(\varepsilon^3). \label{expansion_f_R}
\end{equation}
Finally, using the expansion \eqref{expansion_f_R} in \eqref{2Dsys_R}, we obtain the leading order equation
\begin{equation}
	R_z = c \Big{(} R^2 - \frac{1}{4} \Big{)}, \label{redued_eq}
\end{equation}
where
\begin{equation*}
	c = \frac{\gamma (\gamma + 1) (P^-)^{\gamma - 2}}{2 s \mu} > 0.
\end{equation*}
Equation \eqref{redued_eq} has solutions converging to $\mp 1/2$ as $y \rightarrow \pm \infty$ with $R_z< 0$ that are contained in  $(-1/2,1/2)$.
	
Let $Q = R_z$. The slow manifold (see \cite{J95}) is given by
\begin{equation*}
	M_{\varepsilon} = \left\{(Q,R) \in \R^2 \, : \, Q = h^{\varepsilon}(R) \right\},
\end{equation*}
where
\begin{equation}
	h^{\varepsilon}(R) = h_0(R) + \varepsilon h_1(R,\varepsilon),
\end{equation}
	with
\begin{equation*}
	h_0(R) = c\Big{(} R^2 - \frac{1}{4} \Big{)}.
\end{equation*}
The equation on the slow manifold is
\begin{equation}
	R_z= c \Big{(} R^2 - \frac{1}{4} \Big{)} + \varepsilon h_1(R, \varepsilon). \label{eq_slow_manifold}
\end{equation}
For sufficiently small $\varepsilon > 0$, the term $\varepsilon h_1(R, \varepsilon)$ does not affect the monotonicity of the solutions of \eqref{eq_slow_manifold}, which converge to $\mp 1/2$ as $z \rightarrow \pm \infty$. Moreover, these solutions are strictly decreasing. Then we have
\begin{equation} \label{changevar_y}
P(y) = \varepsilon R(\varepsilon y) + \frac{P^+ + P^-}{2}, \qquad y\in\R.
\end{equation}
Direct differentiation of \eqref{changevar_y} with respect to $y$ gives
\begin{equation}\label{eq_dP}
	P'(y) = \varepsilon^2 R'(\varepsilon y),
\end{equation}
and, since $|R'(z)| \leq C$, where $C$ does not depend on $\varepsilon$, we obtain the first inequality in \eqref{eq:prop-P}.
Moreover, direct differentiation of \eqref{eq_dP} with respect to $y$ and equation \eqref{eq_slow_manifold} yield
\begin{align*}
	P''(y) &= \varepsilon^2 \frac{d}{dy}\big(R'(\varepsilon y)\big) = 
	\varepsilon^2 \frac{d}{dy}\left( c \Big( R(\varepsilon y)^2 - \frac{1}{4} \Big) + \varepsilon h_1(R(\varepsilon y), \varepsilon)\right) \\
	&= \varepsilon^3 R'(\varepsilon y) \left( 2 c R(\varepsilon y)  + \varepsilon \frac{\partial h_1}{\partial R}(R(\varepsilon y), \varepsilon)\right)\\
	&=\e P'(y)\left(2 c R(\varepsilon y)  + \varepsilon \frac{\partial h_1}{\partial R}(R(\varepsilon y), \varepsilon)\right).
\end{align*}
Since
\begin{equation*}
	\left| 2 c R + \varepsilon \frac{\partial h_1}{\partial R} \right| \leq C_2,
\end{equation*}
where the constant $C_2 > 0$ does not depend on $\varepsilon$, we obtain the second inequality in \eqref{eq:prop-P}.

We now prove \textbf{(ii)}. Recall that we have $s<0$, $P^- < P^+$ and $\e := P^+ - P^-$ is the amplitude of the shock. Let us make the change of variables $\tilde{y} = - y$ and $\tilde{P}(\tilde{y}) = P(y)$ in \eqref{2Dsys}. Denoting $' = d/d\tilde{y}$, we obtain the equation
\begin{equation*}
\tilde{P}''=\frac{2}{k^2} f(\tilde{P}) - \frac{2 \tilde{s} \mu}{k^2} \tilde{P}' + \frac{\tilde{P}'^2}{\tilde{P}},
\end{equation*}
where $\tilde{s} = -s$, which has the same form as equation \eqref{2Dsys}. Let $\tilde{P}^+ = P^-$ and $\tilde{P}^- = P^+$. Hence, we repeat the steps in case \textbf{(i)} (with the difference that we now expand around $\tilde{P}^+$) in order obtain the reduced equation
\begin{equation}
	R_{\tilde z}= c \Big{(} R^2 - \frac{1}{4} \Big{)}, \qquad \mbox{ with } \quad c:=\frac{\gamma (\gamma + 1) (\tilde{P}^+)^{\gamma - 2}}{2 \tilde{s} \mu} > 0,
\end{equation}
where $\tilde z=\e\tilde y$.
The conclusion follows exactly as in case \textbf{(i)}.
\end{proof}

As we will see in Section \ref{sec:main}, the monotonicity of the profile $P$ is crucial in the proof of spectral stability for small amplitude dispersive shock profiles.
From Lemma \ref{lem:P-prop}, it is clear that we can consider two different cases:
either $s>0$ and a monotone decreasing profile $P$, or $s<0$ and a strictly increasing $P$. For simplicity in the exposition, from this point on and for the rest of the paper we consider only the case $s > 0$. Our results can be easily extended to case $s<0$ at the expense of extra bookkeeping.

\subsection{Asymptotic behavior}
\label{secasympb}

Let us examine some properties of the asymptotic decay of the profiles. To start with, let us express the constant $A$ appearing in \eqref{eq:J-P} and in \eqref{expr_A} as a function of $P^+$ and $P^-$:
by proceeding as in \cite{LMZ20a} and using the Rankine--Hugoniot conditions \eqref{eq:RH1} and \eqref{eq:RH2}, we deduce the following quadratic equation
$$(J^-)^2-2sP^-J^-+\frac{P^-P^+\left[(P^+)^{\gamma}-(P^-)^{\gamma}-P^-s^2+(P^-)^2s^2/P^+\right]}{P^--P^+}=0,$$
which has exactly two solutions
$$J^-_{1,2}=sP^-\pm\sqrt{\frac{P^-P^+\left[(P^-)^\gamma-(P^+)^\gamma\right]}{P^--P^+}}.$$

It has been proved in \cite{Zhel-preprint} that the solution $J^-_1$ should not be considered, 
because the corresponding shock $(P^\pm,J^\pm_1,s)$ is not an admissible Lax shock for \eqref{eq:Euler}, 
while $(P^\pm,J^\pm_2,s)$ defines a Lax 2--shock for that system. 
As a consequence, the definition \eqref{expr_A} yields
\begin{equation*}
	A=\sqrt{P^- P^+}\sqrt{\frac{(P^-)^{\gamma}-(P^+)^{\gamma}}{P^--P^+}}>0.
\end{equation*}
If we choose $P^+ = P^- - \varepsilon$ with $0 < \varepsilon < P^-$, we have
\begin{equation}\label{expr_A_P}
	A = \sqrt{P^-(P^- - \varepsilon)}\sqrt{\frac{(P^-)^{\gamma}-(P^- - \varepsilon)^{\gamma}}{\varepsilon}}. 
\end{equation}
We deduce that there exists $\varepsilon_0 > 0$ such that the following expansion holds true
\begin{equation}
	A = P^- c_s(P^-) - \frac{\gamma + 1}{4} c_s(P^-) \varepsilon + \mathcal{O}(\varepsilon^2), \label{expansion_A}
\end{equation}
for $0 < \varepsilon < \varepsilon_0$, where the sound speed \eqref{eq:sound-speed} at the left state is $c_s(P^-) = \sqrt{\gamma (P^-)^{\gamma - 1}}$. 
Let us define the functions 
\begin{align}
	f_1(y)&:=\left(s-\frac{A}{P(y)}\right)^2-\gamma P(y)^{\gamma-1}, \label{eq:f1} \\
	f_2(y)&:=-s+2\frac{A}{P(y)}. \label{eq:f2}
\end{align}

As a consequence of Lemma \ref{lem:P-prop} and the expansion \eqref{expansion_A}, we now prove the following lemma containing some useful properties of the functions $f_1$ and $f_2$.

\begin{lem}\label{lem:f_i}
Suppose $P^- > 0$ and $s\in(0,\bar s)$, where
\begin{equation}
\label{defbars}
\bar s:=\min\left\{2c_s(P^-),\left(\frac{\gamma+1}{2}\right) c_s(P^-)\right\}.
\end{equation}
There exists $\e_1>0$ such that if the shock amplitude
\begin{equation}
\label{eq:small-amplitude}
	\e = P^--P^+,
\end{equation} 
satisfies $\e\in(0,\e_1)$ then there exist uniform constants $c_j > 0$ (independent of $\e$) such that 
\begin{subequations}\label{allproperties}
\begin{align}
-c_1^{-1} &\leq f_1(y) \leq - c_1 < 0, &
  c_2 &\geq f_2(y)\geq c_2^{-1} > 0, \label{eq:f_i-sign}\\
c_3^{-1}|P'(y)| &\leq f'_1(y) \leq c_3|P'(y)|,&
c_3^{-1}|P'(y)| &\leq f'_2(y)\leq c_3|P'(y)|, \label{eq:f'_i-order}\\
  |f''_1(y)|&\leq c_4\e|P'(y)|, &
  |f''_2(y)|&\leq c_4\e|P'(y)|, \label{eq:f''_i-order}
\end{align}
\end{subequations}
for all $y\in\R$, where the functions $f_1$ and $f_2$ are defined in \eqref{eq:f1} and \eqref{eq:f2}, respectively.
\end{lem}
\begin{proof}
By using the expansion \eqref{expansion_A}, one has
$$A=P^-c_s(P^-)+O(\e):=A^-+O(\e).$$
By defining
$$F_1(A,P):=\left(s-\frac{A}{P}\right)^2-\gamma P^{\gamma-1},$$
we see that
$$F_1(A^-,P^-)=\left(s-c_s(P^-)\right)^2-c_s(P^-)^2=s^2-2sc_s(P^-)<0,$$
if $s\in(0,\bar s)$.
Therefore, if we choose $\e>0$ sufficiently small in \eqref{eq:small-amplitude} then $f_1(y)<-c_1<0$ for all $y\in\R$ and some $c_1 > 0$. Moreover, since $f_1(y)$ has finite limits as $y \to \pm \infty$ and $P \in [P^{+}, P^{-}]$ we also obtain $- c_1^{-1} \leq f_1(y) \leq -c_1$ for some uniform $c_1\in(0,1)$ and all $y \in \R$. This shows the estimate for $f_1$ in \eqref{eq:f_i-sign}.

On the other hand, since $f'_1(y)=\partial_PF_1(A,P(y))P'(y)$, $P'(y)<0$, for any $y\in\R$ we obtain
\begin{align*}
	\partial_PF_1(A^-,P^-)&=2\left(s-\frac{A^-}{P^-}\right)\frac{A^-}{(P^-)^2}-\gamma(\gamma-1)(P^-)^{\gamma-2}\\
	&=2(s-c_s(P^-))\frac{c_s(P^-)}{P^-}-\gamma(\gamma-1) (P^-)^{\gamma-2}\\
	&=\frac{c_s(P^-)}{P^-}\left[2s-2c_s(P^-)-(\gamma-1)c_s(P^-)\right]<0,
\end{align*}
for $s\in(0,\bar s)$, yielding the estimates for $f'_1$ in \eqref{eq:f'_i-order} provided that $\e>0$ is sufficiently small.

Similarly, for the the function $f_2$ we define 
$$F_2(A,P):=-s+2\frac{A}{P},$$
and since $F_2(A^-,P^-)=2c_s(P^-)-s>0$ for $s\in(0,\bar s)$, we have $f_2(y) \geq c_2^{-1}$ for some $c_2 > 0$ and all $y\in\R$ provided that $\e>0$ is small. By a similar argument to that of $f_1$, we also conclude that $c_2 \geq f_2(y) \geq c_2^{-1}$ for all $y \in \R$ with uniform $c_2 > 1$. This shows the second estimate in \eqref{eq:f_i-sign}. 

Moreover, $f'_2(y)=\partial_PF_2(A,P(y))P'(y)$, with $P'(y)<0$, for any $y\in\R$ and
\begin{align*}
	\partial_P F_2(A^-,P^-)&=-2\frac{A^-}{(P^-)^2}=-\frac{2c_s(P^-)}{P^-}<0.
\end{align*}
Consequently, \eqref{eq:f'_i-order} holds if $\e>0$ is small enough. 

Finally, it suffices to apply inequalities \eqref{eq:prop-P} in order to obtain \eqref{eq:f''_i-order}, provided that $\e$ is sufficiently small. The lemma is now proved.
\end{proof}

\begin{corollary}[exponential decay]
\label{corexpdecay}
For sufficiently small shock amplitudes, $\e = P^- - P^+$, there holds
\[
\begin{aligned}
|P - P^\pm|, |J - J^\pm| &\leq \cO(\e) e^{-\theta \e |y|},\\
\left|\frac{d^{j}P}{dy^j}\right|, \left|\frac{d^{j}J}{dy^j}\right| &\leq \cO(\e^{j+1}) e^{-\theta \e |y|},
\end{aligned}
\]
as $|y| \rightarrow \infty$,
for $j = 1,2,3,$ and some uniform $\theta > 0$.
\end{corollary}
\begin{proof}
The exponential decay follows immediately by standard ODE estimates of heteroclinic connections around hyperbolic end points. The bounds for $P-P^\pm$ and $P'$ follow from \eqref{changevar_y} and \eqref{eq_dP}, whereas the estimates for the subsequent derivatives can be derived by a bootstrapping argument using the profile equations. The bounds for $J$ follow from those for $P$ by the relation \eqref{eq:J-P}.
\end{proof}
%

Another consequence of Lemma \ref{lem:f_i} is the following Corollary, which will be useful later.

\begin{corollary}
\label{corestg}
Let us define the function
\begin{equation}
\label{defogg}
g(y) := - \frac{1}{2} \left[ \frac{f_2(y)}{f_1(y)} - \mu \Big( \frac{1}{f_1(y)}\Big)'\right]', \qquad y \in \R,
\end{equation}
where $' = d/dy$. Then for sufficiently small shock amplitudes, $\e = P^- - P^+ > 0$, we have
\begin{equation}
\label{eq:g-order2}
g(y) \geq \bar{C} |P'(y)| > 0,
\end{equation}
for all $y \in \R$ and some uniform constant $\bar{C} > 0$ (independent of $\e$).
\end{corollary}
\begin{proof}
Upon differentiation,
\[
g= - \frac{1}{2 f_1^3} \Big( f_2' f_1^2 - f_1 f_1' f_2 + \mu f_1 f_1'' - 2 \mu (f_1')^2 \Big).
\]
Set $0 < \e \ll 1$ small enough so that estimates \eqref{allproperties} from Lemma \ref{lem:f_i} hold. Therefore,
\[
-\mu  \big( f_1 f_1'' - 2 (f_1')^2 \big) \leq \mu |f_1''| |f_1| + 2 \mu |f_1'|^2 \leq c_4 c_1^{-1} \mu \e|P'| + 2 \mu c_3^2 |P'|^2 \leq C_\mu \e |P'|,
\]
for some $C_\mu > 0$ independent of $\e$ and for $0 < \e \ll 1$ (where we have used the first inequality in \eqref{eq:prop-P}). Hence we obtain,
\[
\begin{aligned}
g&= \frac{1}{2 |f_1|^3} \Big( |f_2'| f_1^2 + |f_1| |f_1'| |f_2| + \mu f_1 f_1'' - 2 \mu (f_1')^2 \Big)\\
&\geq \frac{1}{2 |f_1|^3} \Big( |f_2'| f_1^2 + |f_1| |f_1'| |f_2| - C_\mu \e |P'| \Big)\\
&\geq \frac{1}{2} c_1^3 \big( c_3^{-1}c_1^2 + c_1 c_3^{-1}c_2^{-1} - C_\mu \e\big) |P'| \\
&\geq \bar{C} |P'|,
\end{aligned}
\]
for some uniform $\bar{C} > 0$ independent of $\e>0$. This shows the result.
\end{proof}

\subsection{Subsonicity}

We conclude this section by showing a condition that implies the property $s\in(0,\bar s)$, needed in Lemma \ref{lem:f_i}. 
At this point we recall that a state $(P, J)$ is \emph{subsonic} (resp. \emph{sonic}) if $|u| < c_s(P)$ (resp. $|u| = c_s(P)$), where $u=J/P$.

\begin{lem}\label{lemma_s}
Assume $P^- > 0$, $s > 0$, $P^+ = P^- - \varepsilon$ and $J^- = s P^- - A$.
Then, there exists $\varepsilon_2 > 0$ such that if $\e\in(0,\e_2)$, one has
\begin{description}[style=unboxed,leftmargin=0cm]
	\item[(i)] if the state $(P^-, J^-)$ is subsonic then $s < 2 c_s(P^-)$;
	\item[(ii)] if $\gamma > 1$ and 
	\begin{equation}
	|u^-| < \frac{\gamma - 1}{2}c_s(P^-), \label{cond_u}
	\end{equation} 
	then $s < \frac{\gamma + 1}{2}c_s(P^-)$.
\end{description}
\end{lem}
\begin{proof}
First, let us prove \textbf{(i)}.
Since the state $(P^-, J^-)$ is subsonic, we have $|u^-| < c_s(P^-)$. This implies, in turn, that $(u^-)^2 < c_s(P^-)^2$. However, 	
\begin{equation}\label{expr_u}	
	u^- = s - \frac{A}{P^-},
\end{equation}
and using the expansion \eqref{expansion_A}, we infer
\begin{equation*}
	(u^-)^2 = (s - c_s(P^-))^2 + \frac{(\gamma + 1)c_s(P^-)(s-c_s(P^-))}{2 P^-}  \varepsilon + \mathcal{O}(\varepsilon^2).
\end{equation*}	
Suppose, by contradiction, that  $s \geq 2 c_s(P^-)$. Then $s - c_s(P^-) \geq c_s(P^-) > 0$ and $(s - c_s(P^-))^2 \geq c_s(P^-)^2 > 0$. Hence,
\begin{equation*}
	(u^-)^2 \geq c_s(P^-)^2 + \frac{c_s(P^-)^2 (\gamma + 1)}{2 P^-} \varepsilon + \mathcal{O}(\varepsilon^2) > c_s(P^-)^2,
\end{equation*}
for any $\e\in(0,\e_2)$, if $\e_2$ is sufficiently small.
Therefore, we get a contradiction and we have $s < 2 c_s(P^-)$.

Now, let us prove \textbf{(ii)}. From \eqref{cond_u}, it follows that
\begin{equation}
	(u^-)^2 < \Big{(}\frac{\gamma - 1}{2}c_s(P^-)\Big{)}^2. \label{cond_u_sqr}
\end{equation}
Substituting \eqref{expansion_A} into \eqref{expr_u} we obtain
\begin{equation*}
	(u^-)^2 = (s - c_s(P^-))^2 - \frac{(\gamma + 1)c_s(P^-)(c_s(P^-) - s )}{2 P^-}\varepsilon + \mathcal{O}(\varepsilon^2).
\end{equation*}
This equality and \eqref{cond_u_sqr} imply
\begin{equation}
	(s - c_s(P^-))^2 < \Big{(}\frac{\gamma - 1}{2}c_s(P^-)\Big{)}^2 + \frac{(\gamma + 1)c_s(P^-)(c_s(P^-) - s )}{2 P^-}\varepsilon +\mathcal{O}(\varepsilon^2). \label{cond_u_sqr1}
\end{equation}
Suppose, by contradiction, that $s \geq \frac{\gamma + 1}{2}c_s(P^-)$. This yields
\begin{equation}
	s - c_s(P^-) \geq \frac{\gamma - 1}{2}c_s(P^-), \label{ineq_s}
\end{equation}
and since $\frac{\gamma - 1}{2}c_s(P^-)>0$, we obtain
\begin{equation}
	\Big{(}\frac{\gamma - 1}{2}c_s(P^-)\Big{)}^2 \leq (s - c_s(P^-))^2 . \label{ineq_s_sqr}
\end{equation}
Using \eqref{cond_u_sqr1} and \eqref{ineq_s_sqr}, we deduce
\begin{equation}\label{ineq1}
	\frac{(\gamma + 1)c_s(P^-)(c_s(P^-) - s )}{2 P^-}\varepsilon +\mathcal{O}(\varepsilon^2)>0. 
\end{equation}
Multiply inequality \eqref{ineq_s} by
\begin{equation*}
	-\frac{(\gamma + 1)c_s(P^-)}{2 P^-} < 0,
\end{equation*}
in order to obtain
\begin{equation}
	\frac{(\gamma + 1)c_s(P^-)(c_s(P^-) - s )}{2 P^-} \leq -(\gamma^2 - 1)\frac{c_s(P^-)^2}{4 P^-} < 0. \label{ineq2}
\end{equation}
Finally, using \eqref{ineq1} and \eqref{ineq2}, we end up with
\begin{equation}
	0 < -(\gamma^2 - 1)\frac{c_s(P^-)^2}{4 P^-}\varepsilon + \mathcal{O}(\varepsilon^2). \label{ineq3}
\end{equation}
For $\varepsilon$ sufficiently small, the right hand-side of \eqref{ineq3} is negative and we have a contradiction. Therefore, we conclude that
\begin{equation*}
	s < \frac{\gamma + 1}{2}c_s(P^-).
\end{equation*}
and the proof is complete.
\end{proof}

\begin{rem}
\label{remgam3}
Note that for $\gamma \geq 3$, the subsonicity condition for the end state $(P^-, J^-)$ readily implies $|u^-| < \frac{\gamma - 1}{2}c_s(P^-)$.
\end{rem}

\section{The spectral stability problem}\label{secspectprob}
In this section we pose the spectral stability problem and recall previous stability results of \cite{LMZ20a} in the context of small dispersive shocks. Consider a solution to \eqref{QHD-L} of the form $(\bar{\rho}, \bar{m}) (x,t) = (P,J)(y) + e^{\lambda t} (\tilde{\rho}, \tilde{m})(y)$, where $y = x-st$ is the variable of translation, the pair $(\tilde{\rho}, \tilde{m}) = (\tilde{\rho}, \tilde{m})(y)$ denotes a perturbation belonging to an appropriate Banach space $X$ defined in terms of the variable $y$, and $\lambda \in \C$ is the spectral parameter (in-time growth rate). Upon substitution and linearization around the profile $(P,J)(y)$ we obtain the following spectral problem
\begin{equation}\label{eq_variable_coeff}
	\cL \begin{pmatrix}
	\tilde{\rho}\\
	\tilde{m} 
	\end{pmatrix} = \lambda \begin{pmatrix}
	\tilde{\rho}\\
	\tilde{m}
\end{pmatrix},
\end{equation}
where $\cL$ is the linearized operator around the wave,
\begin{equation*}
\cL
\begin{pmatrix}
\tilde{\rho}\\
\tilde{m}
\end{pmatrix}
:=
\begin{pmatrix}
    s \tilde{\rho}_y - \tilde{m}_y\\
    s \tilde{m}_y + { \displaystyle{\Big(\frac{J^2}{P^2}\tilde{\rho}\Big)_y - \Big(\frac{2 J}{P}\tilde{m}\Big)_y }}- \gamma (P^{\gamma-1}\tilde{\rho})_y+\mu \tilde{m}_{yy}+\cL_V\tilde{\rho}
    \end{pmatrix},
\end{equation*}
and
\begin{equation*}
\cL_V \tilde{\rho} := \frac{k^2}{2}\tilde{\rho}_{yyy}-2 k^2 \Big{(} (\sqrt{P})_y\Big{(}\frac{\tilde{\rho}}{\sqrt{P}}\Big{)}_y\Big{)}_y,
\end{equation*}
see \cite{LMZ20a} for details on its derivation. 
For stability purposes, we regard $\cL$ as a closed, densely defined operator acting on the (complex) space $X = L^2(\R) \times L^2(\R)$ 
with domain $D(\cL) = H^3(\R) \times H^2(\R)$. 

The following definition is standard in the stability theory of nonlinear waves \cite{KaPro13,He81}.

\begin{definition}[resolvent and spectra]
Let $\cL : X \to Y$ be a closed linear operator, with $X, Y$ Banach spaces and dense domain $D(\cL) \subset X$. The \emph{resolvent} of $\cL$, denoted as $\varrho(\cL)$, is the set of all complex numbers $\lambda \in \C$ such that $\cL - \lambda$ is injective and onto, and $(\cL - \lambda )^{-1}$ is a bounded operator. The \emph{point spectrum} of $\cL$, denoted as $\ptsp(\cL)$, is the set of all $\lambda \in \C$ such that $\cL - \lambda$ is a Fredholm operator with index equal to zero and non-trivial kernel. The \emph{essential spectrum} of $\cL$, denoted as $\ess(\cL)$, is the set of all $\lambda \in \C$ such that either $\cL -\lambda$ is not Fredholm, or it is Fredholm with non-zero index. The \emph{spectrum} of $\cL$ is defined as $\sigma(\cL) = \ess(\cL) \cup \ptsp(\cL)$.
\end{definition}
\begin{rem}
Since $\cL$ is closed then $\varrho(\cL) = \C \backslash \sigma(\cL)$. Moreover, $\ptsp(\cL)$ consists of isolated eigenvalues with finite multiplicity (see, e.g., \cite{KaPro13,Kat80}). 
\end{rem}

Spectral stability of the profiles is defined in terms of the $L^2$-spectrum of the linearized operator around the wave. This corresponds to stability under localized (finite energy) perturbations.
\begin{definition}
\label{defspectralstab}
The dispersive shock profile $(P,J)$ is spectrally stable if the $L^2$-spectrum of the linearized operator around the wave $\cL$ is contained in the stable complex half plane, that is,
\[
\sigma(\cL) \subset \{ \lambda \in \C \, : \, \Re \lambda < 0 \} \cup \{ 0 \}.
\]
\end{definition}

\begin{rem}
As it is customary, $\lambda = 0$ is an eigenvalue of $\cL$ associated to the eigenfunction $(P',J')^\top$. Indeed, from the exponential decay (Corollary \ref{corexpdecay}) it follows that $(P',J')^\top \in D(\cL)$ and, upon substitution, it is easy to verify that $\cL (P',J')^\top = 0$; $\lambda = 0$ is called the translational invariance eigenvalue.
\end{rem}

\subsection{Stability of the essential spectrum}

It is proved in \cite[Lemma 3]{LMZ20a} that the essential spectrum of $\cL$ is stable provided that the end states are subsonic or sonic. The following lemma provides a sufficient condition for the stability of the essential spectrum in the context of small-amplitude shocks.
\begin{lem}
\label{lemessspect}
Assume $P^- > 0$ and $s\in(0,2c_s(P^-))$. If the shock amplitude, $\e = P^--P^+ > 0$, is sufficiently small then the $L^2$-essential spectrum of the linearized operator around the shock profile is stable. More precisely,
\[
\ess(\cL) \subset \{\lambda \in \mathbb{C} \, : \, \Re \lambda < 0\} \cup \{0\}.
\]
\end{lem}

\begin{proof}
Let us define $\delta_1 := s/2$ and $\delta_2 := (2 c_s(P^-) - s)/2$. 
Clearly, one has 
$$\delta_1\in(0,s), \qquad \qquad \delta_2\in(0,2c_s(P^-)-s),$$ 
because of the assumption $s\in(0,2c_s(P^-))$.
Moreover, set 
$$\delta:=\min\{\tilde{\delta}_1,\tilde{\delta}_2\}, \qquad \mbox{ where } \qquad \tilde{\delta}_i:=\delta_i/c_s(P^-),  \quad i=1,2.$$ 
We have $\delta c_s(P^-) \leq \tilde{\delta}_1 c_s(P^-)$ and by using the definition of $\tilde{\delta}_1$, we get 
$$\delta c_s(P^-)\leq\delta_1<s.$$
Similarly, $-\delta c_s(P^-) \geq -\tilde{\delta}_2 c_s(P^-)$ and, as a trivial consequence, 
$$2 c_s(P^-)-\delta c_s(P^-) \geq 2 c_s(P^-) -\tilde{\delta}_2 c_s(P^-)\geq2 c_s(P^-) - \delta_2>s,$$
where we used the definition of $\tilde{\delta}_2$ and that $\delta_2<2c_s(P^-)-s$. 
In summary, we deduced
\begin{equation}\label{eq:s}
	\delta c_s(P^-) < s < (2 - \delta) c_s(P^-).
\end{equation}
We have $\delta = \delta(P^-,s,\gamma)$ and $0 < \delta < 1$. 
From \eqref{eq:s} it follows that
\begin{equation*}
	-(1-\delta) c_s(P^-) < s - c_s(P^-) < (1-\delta) c_s(P^-),
\end{equation*}
which is equivalent to
$$(s - c_s(P^-))^2 <(1 - \delta)^2 c_s(P^-)^2.$$
By using \eqref{expr_u}, the expansion \eqref{expansion_A} and the latter estimate, we infer
\begin{align*}
	(u^-)^2&=(s - c_s(P^-))^2 + \frac{(\gamma + 1)c_s(P^-)(s-c_s(P^-))}{2 P^-}  \varepsilon + \mathcal{O}(\varepsilon^2)\\
	&< (1 - \delta)^2 c_s(P^-)^2 + a_1(\varepsilon),
\end{align*} 
where $|a_1(\varepsilon)| = \mathcal{O}(\varepsilon)$. 
For sufficiently small $\varepsilon > 0$, we have $a_1(\varepsilon) < \delta^2 c_s(P^-)^2$, and then
\begin{equation}\label{ineq_um}
	(u^-)^2 < (1 + 2 \delta(\delta - 1)) c_s(P^-)^2 < c_s(P^-)^2.
\end{equation}
Thus, we have proved that the left state is subsonic. 
Let us now consider the right state: from the Rankine-Hugoniot condition \eqref{eq:RH2}, it follows that
\begin{equation*}
	J^+(\varepsilon) = J^-(\varepsilon) + s(P^+(\varepsilon) - P^-).
\end{equation*}
Dividing by $P^+(\varepsilon)>0$, using $u^+(\varepsilon) = J^+(\varepsilon)/P^+(\varepsilon)$ and $P^+(\varepsilon) = P^- - \varepsilon$, we obtain
\begin{equation*}
	u^+(\varepsilon) = \frac{J^-(\varepsilon)}{P^--\varepsilon} - s \varepsilon.
\end{equation*}
Since $u^-(\varepsilon) = J^-(\varepsilon)/P^-$, we get
\begin{equation*}
u^+(\varepsilon) = \frac{P^-}{P^- - \varepsilon}u^-(\varepsilon) - s \varepsilon.
\end{equation*}
By using the relation
\begin{equation*}
	\frac{P^-}{P^- - \varepsilon} = 1 + \mathcal{O}(\varepsilon),
\end{equation*}
we deduce
\begin{equation*}
	u^+(\varepsilon) = u^-(\varepsilon)+\mathcal{O}(\varepsilon),
\end{equation*}
and
\begin{equation*}
	u^+(\varepsilon)^2 = u^-(\varepsilon)^2 + a_2(\varepsilon),
\end{equation*}
where $|a_2(\varepsilon)| = \mathcal{O}(\varepsilon)$. 
The definition \eqref{eq:sound-speed} implies $c_s(P^+)^2 = c_s(P^-)^2 + a_3(\varepsilon)$, with $|a_3(\varepsilon)| = \mathcal{O}(\varepsilon)$,
and subtracting the last two equations, we infer 
\begin{equation*}
	(u^+)^2 - c_s(P^+)^2 = (u^-)^2 - c_s(P^-)^2 + a_4(\varepsilon),
\end{equation*}
where $a_4(\varepsilon):= a_2(\varepsilon)-a_3(\varepsilon)$, and so, $|a_4(\e)|=\mathcal{O}(\varepsilon)$. 
Choosing $\varepsilon>0$ small enough that, from \eqref{ineq_um}, we get 
$$(u^-)^2 - c_s(P^-)^2 < 2 \delta (\delta - 1) c_s(P^-)^2,$$ 
and $a_4(\varepsilon) < \delta (1-\delta) c_s(P^-)^2$, we conclude that
\begin{equation*}
	(u^+)^2 - c_s(P^+)^2< 2 \delta (\delta - 1) c_s(P^-)^2 +a_4(\varepsilon)<\delta (\delta - 1) c_s(P^-)^2 < 0,
\end{equation*}
that implies $|u^+| < c_s(P^+)$ and the right state is subsonic as well.
Hence, we apply Lemma 3 from \cite{LMZ20a} to conclude that the essential spectrum is contained in $\{\lambda \in \mathbb{C}:\Re(\lambda) < 0\} \cup \{0\}$.
\end{proof}

\begin{rem}
\label{remnogap}
It is important to observe that, as proved in \cite{LMZ20a}, there is accumulation of the essential spectrum near the origin. Indeed, from standard analyses \cite{KaPro13,He81}, it is known that the essential spectrum is sharply bounded to the left of the Fredholm borders, $\Sigma_\pm = \{ \lambda = \lambda_\pm(\xi) \in \C \, : \, \xi \in \R\}$, where $\lambda = \lambda_\pm(\xi)$ is determined by the roots of the dispersion relation (see eqn. (36) in \cite{LMZ20a}),
\begin{equation}
\label{disprel}
\lambda^2 + \xi ( \mu \xi - i(s + \beta^\pm)) \lambda + \xi^2 \Big(\tfrac{1}{2} k^2 \xi^2 - \alpha^\pm - s(\beta^\pm + i \mu \xi) \Big) = 0, \qquad \xi \in \R,
\end{equation}
where $\beta^\pm = s - 2J^\pm/P^\pm$ and $\alpha^\pm = (J^\pm)^2/ (P^\pm)^2 - \gamma (P^\pm)^{\gamma-1}$. From \eqref{disprel} it is clear that there is tangency at the origin. Moreover, it is shown in \cite{LMZ20a} that these curves are contained in the stable half plane. This implies that for $\Re \lambda \geq 0$, $\lambda \neq 0$ the Fredholm index of $\cL - \lambda I$ is zero and consequently $\ess(\cL) \subset \{ \Re \lambda < 0\} \cup \{ 0\}$. See Figure \ref{figFredBrds} for a calculation of these Fredholm borders for particular parameter values. The fact that any (stable) neighborhood of the origin contains elements of the essential spectrum is referred to as the absence of a spectral gap between the spectrum of $\cL$ and the imaginary axis. This fact introduces complications to studying the asymptotic (nonlinear) stability of the profiles as it precludes the application of standard exponentially decaying semigroup theory (see \cite{KaPro13} for further information).
\begin{figure}[t]
\begin{center}
\includegraphics[scale=.55, clip=true]{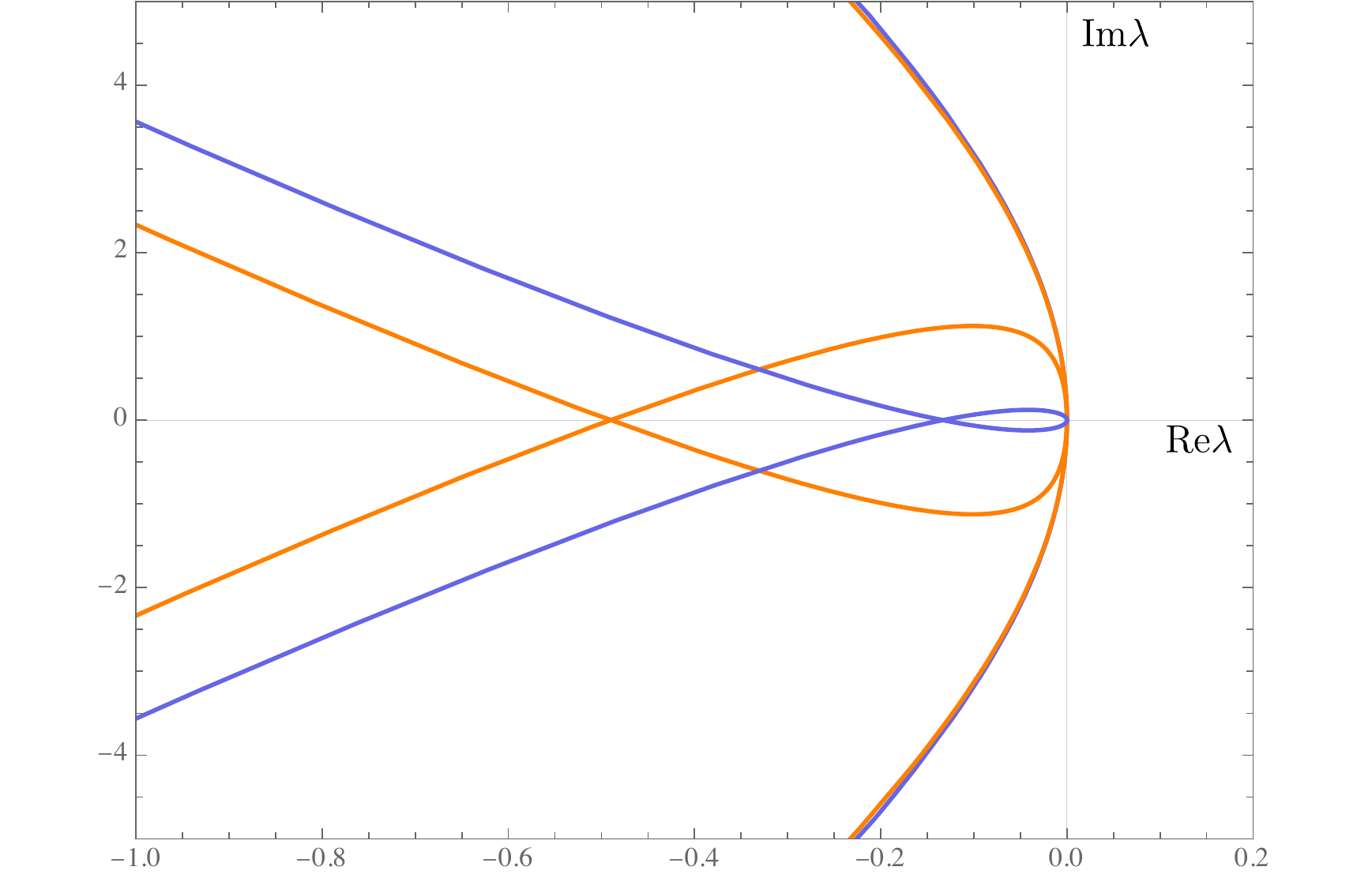}
\end{center}
\caption{\small{Fredholm borders $\Sigma_\pm$ as the complex roots $\lambda = \lambda_+(\xi)$ (orange) and $\lambda = \lambda_-(\xi)$ (blue), with $\xi \in \R$, of the dispersion relation \eqref{disprel}. Here $P^+ = 0.519$, $P^- = P^+ + \e$, $\e = 0.2$, $J^+ = -0.418$, $\gamma = 1.5$, $\mu = 0.1$ and $k = 0.5$.  The essential spectrum of $\cL$ is sharply bounded to the left of these curves in the complex plane (color online).}}
\label{figFredBrds}
\end{figure}
\end{rem}

\subsection{The point spectrum and the integrated operator}

Following Goodman \cite{Go86,Go91}, we now recast the above linearized system in terms of \emph{integrated variables}. This transformation removes the zero eigenvalue without further modifications to the point spectrum (see Lemma \ref{lemequiv} below) and, more importantly, provides better energy estimates. To this end, consider
\begin{equation}
\label{intervar}
\rho(x)=\int_{-\infty}^x \tilde{\rho}(y)dy, \qquad m(x)=\int_{-\infty}^x \tilde{m}(y) dy.
\end{equation} 
Integrating equation \eqref{eq_variable_coeff} it follows that for $\lambda \neq 0$ the integrated variables $\rho$ and $m$ decay exponentially as $|x|\rightarrow \infty$. Expressing $\tilde{\rho}$ and $\tilde{m}$ in terms of $\rho$ and $m$, and integrating \eqref{eq_variable_coeff} from $-\infty$ to $x$ we get the following system in integrated variables,
\begin{equation}
\label{eq:sys-int}
\begin{aligned}
\lambda \rho &= s \rho' - m',\\
\lambda m &= f_1 \rho' + f_2 m' + \mu m''+ {\displaystyle{\frac{k^2}{2}}}\rho''' - 2k^2 \left(\sqrt{P}\right)'\left(\displaystyle\frac{\rho'}{\sqrt{P}}\right)',
\end{aligned}
\end{equation}
where $' = d/dx$ and 
\begin{align*}
	f_1(x)&=\frac{J(x)^2}{P(x)^2}-\gamma P(x)^{\gamma-1}, \\
	f_2(x)&=s-2\frac{J(x)}{P(x)}. 
\end{align*}
Notice that the functions $f_1$ and $f_2$ above are exactly the functions defined in \eqref{eq:f1} and \eqref{eq:f2}; this follows from straightforward algebra using \eqref{eq:J-P} and \eqref{expr_A}.


In view of the form of the integrated system \eqref{eq:sys-int}, we now define the \emph{integrated operator} as,
\[
\begin{aligned}
\cI &: D(\cI) \subset L^2(\R) \times L^2(\R) \longrightarrow L^2(\R) \times L^2(\R),\\
\cI \begin{pmatrix} \rho \\ m \end{pmatrix} &:= 
\begin{pmatrix} 
s \rho' - m' \\ 
f_1 \rho' + f_2 m' + \mu m'' + {\displaystyle{\frac{k^2}{2}}} \rho''' - 2k^2\left(\sqrt{P}\right)'\left(\displaystyle\frac{\rho'}{\sqrt{P}}\right)' \end{pmatrix},
\end{aligned}
\]
with domain $D := D(\cI) = D(\cL) = H^3(\R) \times H^2(\R)$. Once again, the integrated operator is a closed, linear and densely defined operator in $L^2(\R) \times L^2(\R)$. 

\begin{lemma}
\label{lemequiv}
The point spectrum of the original operator $\cL$ is contained in that of the integrated operator $\cI$, except for the eigenvalue zero. More precisely, 
\[
\ptsp(\cL) \backslash \{0\} \subset \ptsp(\cI).
\]
\end{lemma}
\begin{proof}
Suppose $\lambda \in \ptsp(\cL)$, $\lambda \neq 0$ with eigenfunction $W = (\tilde{\rho}, \tilde{m})^\top \in D$ satisfying $\cL W = \lambda W$. Define $\rho = \rho(x)$ and $m = m(x)$ as the antiderivatives in \eqref{intervar} of $\tilde{\rho}$ and $\tilde{m}$, respectively. Then clearly $\rho' \in H^3(\R)$ and $m' \in H^2(\R)$. Moreover, in view that $\lambda \neq 0$, integration of \eqref{eq_variable_coeff} yields the exponential decay as $|x|\rightarrow \infty$ of the integrated variables, $\rho$ and $m$. This shows, in turn, that $\rho, m \in L^2(\R)$. We conclude that $V := (\rho, m)^\top \in D$ is a solution to system \eqref{eq:sys-int}. This implies that $\cI V = \lambda V$ and that $\ptsp(\cL) \backslash \{0\} \subset \ptsp(\cI)$, as claimed.
\end{proof}

As a consequence of Lemma \ref{lemequiv} it suffices to prove that $\ptsp(\cI)$ is stable in order to conclude the same for the original operator $\cL$.

\section{Energy estimates}
\label{sec:main}

In this section we prove the main result of the paper. For that purpose, we establish energy estimates for the point spectrum of the integrated operator $\cI$. Lemma \ref{lem:f_i} and the properties \eqref{allproperties} play a key role along the proof.


\begin{lemma}[energy estimate]
\label{lemeest}
Assume $P^- > 0$ and $s\in(0,\bar s)$ where $\bar{s}> 0$ is defined in \eqref{defbars}. Let the shock amplitude, $\e = P^- - P^+ > 0$, be sufficiently small. Then the integrated operator is point spectrally stable, that is, $\ptsp(\cI) \subset \{ \Re \lambda < 0\}$.
\end{lemma}
\begin{proof}
By contradiction, let us suppose that $\lambda \in \ptsp(\cI)$ with $\Re \lambda \geq 0$ and with eigenfunction $V = (\rho, m)^\top \in D$. Then $\cI V = \lambda V$ and equations \eqref{eq:sys-int} hold. Multiply the second equation in \eqref{eq:sys-int} by $m^*/|f_1|$ and integrate over $\R$. The result is
\[
\begin{aligned}
\lambda \int_{\R} \frac{|m|^2}{|f_1|} \, dx &= \int_{\R} \left[ \frac{f_1}{|f_1|} m^* \rho' + \frac{f_2}{|f_1|} m^*m' + \frac{\mu}{|f_1|} m^* m'' + \frac{k^2}{2|f_1|} m^* \rho'''+\right. \\
&\quad \qquad \left.- 2 k^2 \frac{(\sqrt{P})'}{|f_1|} m^* \!\Big( \frac{\rho'}{\sqrt{P}}\Big)'\right] \, dx.
\end{aligned}
\]
Let us choose $\e = P^- - P^+$ small enough so that estimates \eqref{allproperties} from Lemma \ref{lem:f_i} hold. Then, in particular, we have that $f_1 \leq -c_1 <0$, $f_1/|f_1| = \text{sgn} f_1 = -1$. Also, from the first equation in \eqref{eq:sys-int} we have the identity $m' = s \rho' - \lambda \rho$. Therefore, substitute last identity and integrate by parts to obtain
\begin{equation}
\label{util}
\begin{aligned}
\lambda &\int_{\R} \frac{|m|^2}{|f_1|} \, dx - \int_{\R} (s\rho' - \lambda \rho)^* \rho \, dx + \int_\R \left[ \frac{f_2}{f_1} - \mu \Big(\frac{1}{f_1} \Big)'\right] m^* m' \, dx + \mu \int_\R \frac{|m'|^2}{|f_1|}\, dx = \\
&= - \frac{k^2}{2} \int_\R \frac{(m')^*\rho''}{|f_1|} \, dx + \frac{k^2}{2} \int_\R \Big(\frac{1}{f_1} \Big)' m^* \rho'' \, dx - 2k^2 \int_\R \frac{(\sqrt{P})'}{|f_1|} m^* \!\Big( \frac{\rho'}{\sqrt{P}}\Big)' \, dx.
\end{aligned}
\end{equation}
Since
\[
\begin{aligned}
\Re \! \int_{\R} (s\rho' - \lambda \rho)^* \rho \, dx &= s \, \Re \! \int_\R (\rho')^* \rho \, dx - (\Re \lambda) \int_\R |\rho|^2 \, dx \\
&= \frac{s}{2} \int_\R \frac{d}{dx} \big( |\rho|^2\big) \, dx - (\Re \lambda) \int_\R |\rho|^2 \, dx \\
&=- (\Re \lambda) \int_\R |\rho|^2 \, dx,
\end{aligned}
\]
and
\[
\begin{aligned}
\Re \! \int_{\R} \left[ \frac{f_2}{f_1} - \mu \Big(\frac{1}{f_1} \Big)'\right] m^* m' \, dx &= \int_{\R} \left[ \frac{f_2}{f_1} - \mu \Big(\frac{1}{f_1} \Big)'\right] \Re (m^* m') \, dx\\
&= \frac{1}{2} \int_\R \left[ \frac{f_2}{f_1} - \mu \Big(\frac{1}{f_1} \Big)'\right] \frac{d}{dx} \big( |m|^2 \big) \, dx\\
&= - \frac{1}{2} \int_\R \left[ \frac{f_2}{f_1} - \mu \Big(\frac{1}{f_1} \Big)'\right]' |m|^2 \, dx,
\end{aligned}
\]
then taking the real part of \eqref{util} and after integration by parts we arrive at
\begin{equation}\label{util2}
\begin{aligned}
(\Re \lambda) &\int_\R \left[ |\rho|^2 + \frac{|m|^2}{|f_1|}\right] \, dx + \int_\R g |m|^2 \, dx + \mu \int_\R \frac{|m'|^2}{|f_1|}\, dx = \\
&\quad = \frac{k^2}{2} \Re \! \int_\R \left[ \frac{(m')^* \rho''}{f_1}  + \Big( \frac{1}{f_1} \Big)' m^* \rho'' + \frac{4(\sqrt{P})'}{f_1} m^* \!\Big( \frac{\rho'}{\sqrt{P}}\Big)' \right] \, dx,
\end{aligned}
\end{equation}
where $g = g(x)$ is the function defined in \eqref{defogg} and where we have substituted $|f_1|= - f_1$.

Let us integrate by parts the terms involved in the right hand side of \eqref{util2}. First, use again the identity $m' = s\rho' - \lambda \rho$ from the first equation in \eqref{eq:sys-int} to obtain
\begin{align}
\Re \! \int_\R \frac{(m')^* \rho''}{f_1} \, dx &= - \Re \! \int_\R \frac{(m'')^* \rho'}{f_1} \, dx - \Re \! \int_\R \Big( \frac{1}{f_1} \Big)' (m')^* \rho' \, dx \nonumber\\
&= - \Re \! \int_\R \frac{(s \rho'' - \lambda \rho' )^* \rho'}{f_1} \, dx - \Re \! \int_\R \Big( \frac{1}{f_1} \Big)' (m')^* \rho' \, dx \nonumber\\
&= - \frac{s}{2} \int_\R \frac{(|\rho'|^2)'}{f_1} \, dx + (\Re \lambda) \int_\R \frac{|\rho'|^2}{f_1} \, dx - \Re \! \int_\R \Big( \frac{1}{f_1} \Big)' (m')^* \rho' \, dx \nonumber\\
&= \frac{s}{2} \int_\R \Big( \frac{1}{f_1} \Big)'  |\rho'|^2 \, dx - (\Re \lambda) \int_\R \frac{|\rho'|^2}{|f_1|} \, dx - \Re \! \int_\R \Big( \frac{1}{f_1} \Big)' (m')^* \rho' \, dx \nonumber\\
&= - \frac{s}{2} \int_\R \frac{|f_1'|}{|f_1|^2} |\rho'|^2 \, dx - (\Re \lambda) \int_\R \frac{|\rho'|^2}{|f_1|} \, dx - \Re \! \int_\R \Big( \frac{1}{f_1} \Big)' (m')^* \rho' \, dx.\label{Thet1}
\end{align}
Moreover,
\begin{equation}
\label{Thet2}
\Re \! \int_\R \Big( \frac{1}{f_1} \Big)' m^* \rho'' \, dx = - \Re \! \int_\R \Big( \frac{1}{f_1} \Big)' (m')^* \rho' \, dx -\Re \! \int_\R \Big( \frac{1}{f_1} \Big)'' m^* \rho' \, dx,
\end{equation}
and,
\begin{equation}
\label{Thet3}
\begin{aligned}
\Re \! \int_\R \frac{(\sqrt{P})'}{f_1} m^* \!\Big( \frac{\rho'}{\sqrt{P}}\Big)' \, dx &= - \Re \! \int_\R \frac{(\sqrt{P})''}{f_1 \sqrt{P}} m^* \rho' \, dx - \Re \! \int_\R \frac{(\sqrt{P})'}{f_1 \sqrt{P}} (m')^* \rho' \, dx + \\
& \quad - \Re \! \int_\R \frac{(\sqrt{P})'}{\sqrt{P}} \Big( \frac{1}{f_1} \Big)' m^* \rho' \, dx.
\end{aligned}
\end{equation}
Substitute \eqref{Thet1}, \eqref{Thet2} and \eqref{Thet3} into \eqref{util2}. This yields
\begin{equation}
\label{labuena}
\begin{aligned}
(\Re \lambda) \int_\R \left[ |\rho|^2 + \frac{|m|^2}{|f_1|} + \frac{k^2}{2} \frac{|\rho'|^2}{|f_1|} \right] \, dx &+ \int_\R g |m|^2 \, dx + \mu \int_\R \frac{|m'|^2}{|f_1|}\, dx + \\ &+ \frac{k^2 s}{4} \int_\R \frac{|f_1'|}{|f_1|^2} |\rho'|^2 \, dx = \Theta,
\end{aligned}
\end{equation}
where
\begin{equation}
\label{defTheta}
\begin{aligned}
\Theta &:= - \frac{k^2}{2} \Re \big( I_1 + I_2 \big),\\
I_1 &:= \int_\R \left[ 2 \Big( \frac{1}{f_1} \Big)' + \frac{(\sqrt{P})'}{f_1 \sqrt{P}} \right] (m')^* \rho' \, dx,\\
I_2 &:= \left[ \Big( \frac{1}{f_1} \Big)'' + \frac{(\sqrt{P})'}{\sqrt{P}} \Big(\frac{1}{f_1} \Big)' +  \frac{(\sqrt{P})''}{f_1 \sqrt{P}} \right] m^* \rho' \, dx.
\end{aligned}
\end{equation}

For $0 < \e \ll 1$ sufficiently small, apply Lemma \ref{lem:f_i} and Corollary \ref{corestg} to deduce from \eqref{labuena} the following estimate
\begin{equation}
\label{labuena2}
\begin{aligned}
(\Re \lambda) \int_\R \left[ |\rho|^2 + \frac{|m|^2}{|f_1|} + \frac{k^2}{2} \frac{|\rho'|^2}{|f_1|} \right] \, dx &+ \bar{C}\int_\R |P'| |m|^2 \, dx + \mu c_1 \int_\R |m'|^2 \, dx + \\ &+ \frac{c_1^2 k^2 s}{4 c_3} \int_\R |P'| |\rho'|^2 \, dx \leq \Theta.
\end{aligned}
\end{equation}

In this fashion we have gathered all the terms with a definite sign in the left hand side of estimate \eqref{labuena2}. Next, we show that all the terms in \eqref{defTheta} can be absorbed into the left hand side provided $\e > 0$ is sufficiently small. In the sequel, $C > 0$ denotes a uniform positive constant, independent of $\e, \lambda$ and of the perturbation variables (which may depend on the parameters $\mu, k, s > 0$), whose value may change from line to line.

Notice that
\[
\Big(\frac{1}{f_1} \Big)' = - \frac{f_1'}{f_1^2} = - \frac{|f_1'|}{|f_1|^2} < 0,
\]
and as a consequence,
\begin{equation}\label{lados}
\left| \Big(\frac{1}{f_1} \Big)' \right| \leq \frac{c_3}{c_1^2} |P'| \leq C |P'|,
\end{equation}
because of \eqref{allproperties}; moreover, 
\begin{equation}
\label{latres}
\left| \Big(\frac{1}{f_1} \Big)'' \right| =\left| - \frac{f_1''}{f_1^2} + 2 \frac{(f_1')^2}{f_1^3} \right| \leq \frac{c_4 \e}{c_1^2} |P'| + \frac{2 c_3^2}{c_1^3}|P'|^2 \leq C \e |P'|,
\end{equation}
for all $x \in \R$ and some $C > 0$, in view that $P' = \cO(\e^2)$. Likewise, from Lemma \ref{lem:f_i} and estimates \eqref{eq:prop-P}, we have the bounds
\begin{equation}
\label{lacuatro}
\begin{aligned}
|(\sqrt{P})'| = \frac{1}{2} \frac{|P'|}{\sqrt{P}} \leq \frac{1}{2} \frac{|P'|}{\sqrt{P^+}} &\leq C |P'|,\\
\left| \frac{(\sqrt{P})'}{f_1 \sqrt{P}} \right| \leq \frac{C|P'|}{c_1 \sqrt{P^+}} &\leq C |P'|,\\
\left| \frac{(\sqrt{P})''}{f_1 \sqrt{P}} \right| \leq \frac{1}{2} \frac{|P''|}{|f_1 P|} + \frac{1}{4} \frac{|P'|^2}{|f_1| P^2} &\leq C \e |P'|.
\end{aligned}
\end{equation}
Apply \eqref{lados} and \eqref{lacuatro}, as well as Young's inequality, to obtain the estimate
\begin{equation}
\label{eta1}
\begin{aligned}
\left| \frac{k^2}{2} \Re I_1 \right| &\leq C \int_\R \left( 2 \left| \Big( \frac{1}{f_1} \Big)' \right|
 + \left| \frac{(\sqrt{P})'}{f_1 \sqrt{P}} \right| \right) |m'| |\rho'| \, dx \\
 &\leq C \int_\R |P'| |m'||\rho'| \, dx \\
 &\leq \frac{C}{4 \eta_1} \int_\R |P'| |m'|^2 \, dx + C \eta_1 \int_\R |P'| |\rho'|^2 \, dx \\
 &\leq \frac{C \e^2}{4\eta_1} \int_\R |m'|^2 \, dx + C \eta_1 \int_\R |P'| |\rho'|^2 \, dx,
 \end{aligned}
 \end{equation}
 for any $\eta_1 > 0$, inasmuch as $|P'|\leq C \e^2$ (see \eqref{eq:prop-P}). In the same fashion, apply \eqref{latres} and \eqref{lacuatro} to get
 \begin{equation}
\label{eta2}
\begin{aligned}
\left| \frac{k^2}{2} \Re I_2 \right| &\leq C \int_\R \left( \left| \Big( \frac{1}{f_1} \Big)'' \right|
 + \left| \frac{(\sqrt{P})'}{\sqrt{P}} \Big( \frac{1}{f_1} \Big)' \right| +  \left| \frac{(\sqrt{P})''}{f_1 \sqrt{P}} \right|\right) |m| |\rho'| \, dx \\
 &\leq C\e \int_\R |P'| |m||\rho'| \, dx \\
 &\leq \frac{C\e}{4 \eta_2} \int_\R |P'| |m|^2 \, dx + C \e \eta_2 \int_\R |P'| |\rho'|^2 \, dx,
 \end{aligned}
 \end{equation}
 for any $\eta_2 > 0$. 
 Now substituting \eqref{eta1} and \eqref{eta2} into \eqref{labuena2}, we infer
 \begin{equation}
 \label{energyest}
 \begin{aligned}
(\Re \lambda) \int_\R \left[ |\rho|^2 + \frac{|m|^2}{|f_1|} + \frac{k^2}{2} \frac{|\rho'|^2}{|f_1|} \right] \, dx &+ \left( \bar{C} - \frac{C\e}{4 \eta_2}\right)\int_\R |P'| |m|^2 \, dx + \\
&+ \left( \mu c_1 - \frac{C \e^2}{4 \eta_1}\right) \int_\R |m'|^2 \, dx + \\ &+ \left( \frac{c_1^2 k^2 s}{4 c_3} - C\eta_1 - C\e \eta_2 \right)\int_\R |P'| |\rho'|^2 \, dx \\&\leq 0.
\end{aligned}
 \end{equation}
Choose $\eta_1 = \e$ and $\eta_2 = \sqrt{\e}$. Then estimate \eqref{energyest} is of the form
\[
 \begin{aligned}
(\Re \lambda) \int_\R \left[ |\rho|^2 + \frac{|m|^2}{|f_1|} + \frac{k^2}{2} \frac{|\rho'|^2}{|f_1|} \right] \, dx &+ \left( \bar{C} - \cO(\sqrt{\e}) \right)\int_\R |P'| |m|^2 \, dx + \\
&+ \big( \mu c_1 - \cO(\e) \big) \int_\R |m'|^2 \, dx + \\ &+ \left( \frac{c_1^2 k^2 s}{4 c_3} - \cO(\e) \right)\int_\R |P'| |\rho'|^2 \, dx \\&\leq 0.
\end{aligned}
\]
This is a contradiction with $\Re \lambda \geq 0$ if $0 < \e \ll1 $ is small enough. We conclude that $\Re \lambda < 0$ and the lemma is proved.
\end{proof}

\begin{theorem}[spectral stability]
\label{mainthm}
Assume $P^- > 0$ and $s\in(0,\bar s)$ where
\begin{equation*}
\bar s = \min\left\{2c_s(P^-),\left(\frac{\gamma+1}{2}\right) c_s(P^-)\right\}.
\end{equation*}
There exists $\bar\e>0$ such that if the shock amplitude, $\e = P^- - P^+$, satisfies $\e\in(0,\bar\e)$, then the dispersive shock profile is spectrally stable.
%
\end{theorem}
\begin{proof}
Under the condition $s \in (0, \bar{s})$, choose $\bar{\e} > 0$ sufficiently small such that the conclusions of Lemmata \ref{lem:P-prop}, \ref{lem:f_i}, \ref{lemma_s}, \ref{lemessspect} and \ref{lemeest}, as well as Corollaries \ref{corexpdecay} and \ref{corestg}, hold. Then from Lemmata \ref{lemequiv} and \ref{lemeest} we conclude the stability of the point spectrum for the linearized operator around the wave, $\ptsp(\cL) \subset \{ \Re \lambda < 0 \} \cup \{0\}$. Combined with the stability of the essential spectrum (Lemma \ref{lemessspect}), we obtain the result for sufficiently weak dispersive shocks.
\end{proof}

\begin{rem}
Let us make some comments on the condition $s\in(0,\bar s)$, which is the only assumption needed 
to prove spectral stability of small amplitude dispersive shock profiles.
First of all, let us rewrite the constant $\bar s$ as
$$\bar s:=
\begin{cases}
	\displaystyle\left(\frac{\gamma+1}{2}\right)c_s(P^-), \qquad &\mbox{ if } \, 1\leq\gamma<3,\\
	2c_s(P^-), & \mbox{ if } \, \gamma\geq3.
\end{cases}
$$
Hence, if $\gamma\geq3$, then the condition
\begin{equation}\label{eq:phys}
	s\in(0,2c_s(P^-))
\end{equation} 
implies stability of both the essential (see Lemma \ref{lemessspect} and \cite{LMZ20a}) 
and the point spectra (see Theorem \ref{mainthm});
while if $\gamma\in[1,3)$, then we need a stronger condition to prove that there are no eigenvalues with strictly positive real part,
that is 
\begin{equation}\label{eq:unphys}
	s\in\left(0,\displaystyle\left(\frac{\gamma+1}{2}\right)c_s(P^-)\right).
\end{equation} 
It is worth noticing that condition \eqref{eq:phys} has a specific physical meaning:
as we proved in Lemmata \ref{lemma_s} and \ref{lemessspect}, it is equivalent to
$$|u^\pm|<c_s(P^\pm),$$
that is, the end states $(P^\pm,J^\pm)$ are subsonic.
On the other hand, to the best of our knowledge condition \eqref{eq:unphys} does not have any particular physical meaning;
for instance, in Lemma \ref{lemma_s} we proved that the following condition on the velocity of the end state $(P^-,J^-)$,
$$|u^-|<\frac{\gamma-1}{2}c_s(P^-),$$
implies \eqref{eq:unphys}, and the latter condition on $u^-$ is very restrictive when $\gamma\to1^+$.
However, even if condition \eqref{eq:unphys} is instrumental in the proof of Theorem \ref{mainthm},
we conjecture that it is not necessary for stability of small amplitude shock profiles.
\end{rem}

\section{Discussion and open problems}


In this paper we have proved the conjecture by Lattanzio \emph{et al.} \cite{LMZ20a,LMZ20b} that subsonic viscous-dispersive shocks for the QHD system with linear viscosity \eqref{QHD-L} are spectrally stable in the small-amplitude regime. Small viscous-dispersive shocks comply with the compressivity of the shock, that is, they remain monotone, and exhibit exponential decay, sharing in this fashion important features with purely viscous shocks in fluid dynamics. In this work we exploit these properties in order to rigorously prove that the $L^2$-spectrum of the linearized operator around a small amplitude dispersive shock remains in the stable complex half plane. For that purpose, we implemented a novel energy estimate that handles the Bohm potential appearing in the nonlinear dispersive term. This contrasts with previous results with constant capillarity (see, e.g., \cite{Hu09}). Our result is compatible with the numerical evidence based on calculations of the associated Evans function presented in \cite{LMZ20b}. 

A natural question that remains open is whether small-amplitude dispersive shocks in QHD are also \emph{nonlinearly} stable. Up to our knowledge, the only work addressing the nonlinear stability of small amplitude, monotone viscous-dispersive shocks for compressible fluids is that of Zhang \emph{et al.} \cite{ZLY16}, for the particular case of the Navier-Stokes-Korteweg system with constant viscosity and capillarity coefficients. Thus, we believe that the study of the effects on stability of the nonlinear dispersive term of Bohmian type that appears in \eqref{QHD-L} is worth pursuing. 

It is to be observed that, according to the numerical calculations by Lattanzio \emph{et al.} \cite{LMZ20a,LMZ20b}, larger amplitude (and hence, oscillatory) dispersive profiles are also spectrally stable. Therefore, an important open problem is to analytically prove that spectral stability of dispersive shocks for system \eqref{QHD-L} holds beyond the small-amplitude regime. When the shock amplitude increases and the dispersive (or capillarity) coefficient plays a more significant role, the profiles exhibit oscillatory behavior. In this case, the monotonicity property no longer holds and the method of proof should change substantially (for a related discussion, see \cite{Hu09}). This is a problem that, because of its difficulty, warrants further investigations.
%

\section*{Acknowledgements}
The authors are grateful to Corrado Lattanzio for useful conversations. The work of D. Zhelyazov was supported by a Post-doctoral Fellowship by the Direcci\'{o}n General de Asuntos del Personal Acad\'{e}mico (DGAPA), UNAM. The work of R. G. Plaza was partially supported by DGAPA-UNAM, program PAPIIT, grant IN-104922.







\def\cprime{$'$}\def\cprime{$'$}\def\cprime{$'$}

\end{document}